\documentclass[12pt]{article}

\usepackage{amssymb, amsthm, amsmath, amsfonts, latexsym, epsfig, color} 

\usepackage[margin=1in]{geometry}

\newtheorem{theorem}{Theorem}
\newtheorem{lemma}{Lemma}

\theoremstyle{definition}

\newtheorem{conjecture}{Conjecture}

\usepackage{hyperref, framed}

\newcommand{\ignore}[1]{}

\title{Combinatorial games played randomly: Chomp and nim}
\author{Pat Devlin\footnote{Swarthmore College, Swarthmore PA, USA \qquad \texttt{pdevlin2@swarthmore.edu}} \and Paulina Trifonova\footnote{Swarthmore College, Swarthmore PA, USA \qquad \texttt{ptrifon1@swarthmore.edu}}}

\date{January 30, 2024}

\begin{document}

\maketitle

\renewcommand{\thefootnote}{\fnsymbol{footnote}}

\setcounter{footnote}{0}
\renewcommand{\thefootnote}{\arabic{footnote}}

\abstract{
In this note, we investigate combinatorial games where both players move randomly (each turn, independently selecting a legal move uniformly at random).  In this model, we provide closed-form expressions for the expected number of turns in a game of Chomp with any starting condition.  We also derive and prove formulas for the win probabilities for any game of Chomp with at most two rows.  Additionally, we completely analyze the game of nim under random play by finding the expected number of turns and win probabilities from any starting position.
}

\section{Introduction}
Combinatorial games are a popular object of study in the field of discrete math.  Typically, the primary research goal when studying combinatorial games is to understand how a game ought to proceed under optimal play and to categorize which moves are optimal.  See \cite{berlekamp2004winning} for a comprehensive text on the general theory of combinatorial games, and see \cite{demaine2001playing} for a general discussion of algorithmic aspects of the topic.

In this paper, we study combinatorial games where both players move randomly: that is, each turn a player independently selects their move uniformly at random from among all legal options.  Incorporating randomness into otherwise deterministic games is not new: earlier researchers have studied games where the next player is chosen randomly \cite{celis2010local} as well as games where one player (but not all) acts randomly \cite{krawec2012analyzing,krawec2015n} as a proxy for novice skill level.  The idea of studying games where both players move randomly has also been explicitly discussed by Beck \cite{beck1996foundations, beck2005positional}, who used this random model to study a class of \textit{maker-breaker games} related to Ramsey theory \cite{hefetz2014positional}.  That said, although the topic of games where both players play randomly is a natural question (amounting to studying a random walk on the game graph), this topic has yet to be investigated for a variety of well-known games.  In this paper we address this by studying the classic games of Chomp and nim.

Chomp is a famous combinatorial game, which was first introduced by Gale \cite{gale1974curious} and popularized in Martin Gardner's \textit{Mathematical Games} column \cite{sciAm}.  In Chomp, discrete cells are arranged into a board, and players take turns selecting a cell to ``chomp."  When a cell is chomped, the board is updated so as to remove every cell except those that are either above or to the left of the selected cell.  That is, if a cell $(x,y)$---in row $x$ and column $y$---is selected, then this removes every cell $(p,q)$ where $p \geq x$ and $q \geq y$.  Gameplay continues until a player selects the uppermost leftmost cell, thus losing the game.

One of the most intriguing aspects of Chomp is that when starting with a rectangular board, it is relatively easy to see that the first player has a winning move by appealing to an elegant strategy-stealing argument.  However, finding a general winning strategy for boards with more than 2 rows is still an open question.  To quote \cite{soltys2011complexity}, ``this is what makes Chomp so curious: it is easy to establish [for rectangular boards] that player 1 can win; what is computationally difficult is to establish how."

Previous research has categorized winning moves for boards having at most 2 rows as well as boards whose second row has only one cell, and partial progress has been made on configurations involving three rows \cite{zeilberger2001three, sun2002improvements,brouwer2005three} as well as configurations containing two rows of size $n$ and a fixed number of additional cells \cite{byrnes2003poset, zeilberger2004chomp}.  Additional work has revealed a surprising connection from physics using a renormalization technique to predict which moves are optimal \cite{friedman2007nonlinear}.  Closer to our work, researchers have also studied versions of Chomp where one player (but not all) moves randomly \cite{krawec2012analyzing, krawec2015n}.

\subsection*{Our Results}
We study the setting where Chomp is played with two random players.  For notation, we number the rows and columns beginning with the top-left cell $(1,1)$.  Given a Chomp board $B$, we say $(x,y) \in B$ iff $B$ contains a cell in row $x$ and column $y$.  Thus, if a player chooses to chomp a cell at $(x,y)$, then doing so would remove every cell $(p,q)$ that satisfies $x \leq p$ and $y \leq q$.  With this, we state our first result.

\begin{theorem}\label{chomp:expected}
A game of random Chomp starting starting from $B$ has expected length
\[
\mathbb{E}[\text{\# turns starting from $B$}] = \sum_{(x,y) \in B} \frac{1}{xy}, 
\]
where the summation is taken over all cells $(x,y)$ in the board $B$.

In particular, if $B$ is an $m \times n$ rectangle, then the expected length is $\sum_{j=1}^{m} \frac{1}{j} \times \sum_{j=1}^{n} \frac{1}{j}$.
\end{theorem}

The above formula conveniently splits as a sum over the individual cells in the starting board, which allows us to precisely characterize the starting boards that take the most (or fewest) expected number of turns given a fixed number of cells.

\begin{theorem}\label{chomp:extremal shapes}
If $B$ is a Chomp board with $N$ cells, then we have
\[
\sum_{j=1} ^{N} \dfrac{1}{j} \leq \mathbb{E}[\text{\# turns starting from $B$}] \leq (1/2 + o(1))\left( \sum_{j=1} ^{N} \dfrac{1}{j} \right)^2,
\]
where $o(1)$ denotes a term tending to $0$ as $N \to \infty$.  Moreover, there are boards having $N$ cells attaining these lower and upper bounds.
\end{theorem}

Determining the likelihood that a given player wins is considerably more nuanced.  To this end, we prove the following, addressesing boards with at most 2 rows.
\begin{theorem}\label{chomp:probability}
For integers $0 \leq k \leq n$, let $P(n,k)$ denote the probability that Player 1 wins starting from a board having $n$ cells in its first row and $k$ in its second.  Then for all $n \geq 0$, we have
\[
P(n,0) = \begin{cases}
1 \qquad &\text{if $n=0$},\\
0 \qquad &\text{if $n=1$},\\
1/2 \qquad &\text{if $n\geq 2$.}
\end{cases}
\]

Furthermore, recursively define two integer sequences $\alpha_k$ and $\beta_k$ as follows.  As base cases, $\alpha_1 = 1$ and $\beta_1 = -1$, and for $k \geq 1$, we define
\begin{eqnarray*}
\alpha_{k+1} &=& (4k^2 + k) \alpha_k + \beta_k\\
\beta_{k+1} &=& -k(k+1) \alpha_k + (4k^2 -k-1) \beta_k.
\end{eqnarray*}
We have $P(1,1) = 1/2$, and for all $n \geq k \geq 1$, if $(n,k) \neq (1,1)$ then
\[
P(n,k) = \dfrac{1}{2} - \dfrac{n \alpha_k + \beta_k}{(2(k-1))! \cdot (n+k)(n+k-1)(n+k-2)}.
\]
\end{theorem}
At the heart of the above theorem is a mutual recurrence relation involving the sequences $\alpha$ and $\beta$, which allows us to find the win probabilities for any Chomp board having at most two rows.\footnote{Note the seemingly exceptional value $P(1,1)$ actually arises simply because otherwise the fraction would be of the form $0/0$.}  If desired, the mutual recurrence defining $\alpha$ and $\beta$ can be neatly written as the matrix-vector equation
\[
\begin{bmatrix}
\alpha_{k+1}\\
\beta_{k+1}
\end{bmatrix} = \begin{pmatrix}
k(4k+1) & 1\\
-k(k+1) & 4k^2 - k -1
\end{pmatrix}\begin{bmatrix}
\alpha_{k}\\
\beta_{k}
\end{bmatrix}
\]
with initial condition $\begin{bmatrix}
\alpha_{1}\\
\beta_{1}
\end{bmatrix} = \begin{bmatrix}
1\\
-1
\end{bmatrix}$.
By computing terms of these sequences as in Table \ref{table:sequence values}, we see for instance that $\alpha_4 = 2576$ and $\beta_4 = -3248$, thus for all $n \geq 4$
\[
P(n,4) = \dfrac{1}{2} - \dfrac{2576 n -3248}{6! \cdot (n+4)(n+3)(n+2)}. 
\]

\begin{table}[h]\label{table:sequence values}
\begin{center}
\begin{tabular}{c||c|c}
$k$ & $\alpha_k$ & $\beta_k$\\ \hline \hline
1 & 1 & -1\\ \hline
2 & 4 & -4\\ \hline
3 & 68 & -76\\ \hline
4 & 2576 & -3248\\ \hline
5 & 171920 & -243152\\ \hline
6 & 17808448 & -28013888\\ \hline
7 & 2643253312 & -4585857472\\ \hline
8 & 531994564864 & -1010163390208
\end{tabular}
\caption{Values of $\alpha_k$ and $\beta_k$ computed from their definition}
\end{center}
\end{table}

Our proof of Theorem \ref{chomp:probability} is essentially an induction argument, but the motivation behind the sequence definitions---and the courage to suspect such a formula might be true in the first place---took some time for the authors to manifest.  Although we spare the reader the details, the result was discovered by first computing many values of $P(n,k)$ and then experimentally finding formulas for $P(n,0)$, $P(n,1)$, \ldots, and $P(n,8)$.  Noticing the form $P(n,k) = 1/2- (A_k n + B_k) / ((n+k)(n+k-1)(n+k-2))$, we conjectured such a result holds without yet knowing how the coefficients $A_k$ and $B_k$ might be defined.  Supposing the truth of our conjecture, we were then able to derive the relevant recurrence relations that would need to be true.  A glimpse of this can be seen in the proof of our Lemma \ref{lemma:induction}, where the truth of the final sentence is equivalent to the validity of the recursive definitions for $\alpha$ and $\beta$ as given above.

In order to better understand our formula for $P(n,k)$, we turn our attention to the sequences $\alpha_k$ and $\beta_k$.

\begin{theorem}\label{chomp:corollary}
With notation from Theorem \ref{chomp:probability}, for all $k \geq 2$ we have
\[
0 < \alpha_k \leq -\beta_k \leq (k-1) \alpha_k.
\]
Moreover, for all $k \geq 1$ we have $4^{k-1} [(k-1)!]^2 \leq \alpha_{k} \leq k 4^{k-1} [(k-1)!]^2$.

As a consequence, for all $n \geq k \geq 0$, if $(n,k) \neq (0,0)$ then we have
\[
P(n,k) \leq \dfrac{1}{2}.
\]
And furthermore, if $n \geq k \geq 1$ and $(n,k) \neq (1,1)$ then we have
\[
P(n,k) \leq \dfrac{1}{2} - \dfrac{(n+1-k) \sqrt{k}}{(n+k)(n+k-1)(n+k-2)} < \dfrac{1}{2}
\]
and
\[
P(n,k) \geq \dfrac{1}{2} - \dfrac{2n k^{1.5}}{(n+k)(n+k-1)(n+k-2)}.
\]
\end{theorem}

Interestingly, this shows that for any non-empty Chomp board with at most two rows, the first player wins with probability at most 1/2.  If we further assume that the board has more than one row and more than one column, then this probability is strictly less than 1/2.  Thus, although a game of $m \times n$ Chomp is a win for the starting player under optimal play, if there are at most two rows, then it is a disadvantage to go first under random play.  This motivates the following conjectures.

\begin{conjecture}
Suppose Chomp is to be randomly played with a starting configuration of more than one row and more than one column.  The player going first will win with probability strictly less than $0.5$.
\end{conjecture}
This conjecture holds for every configuration we calculated, but proving it would presumably involve some level of insight into how quickly these probabilities approach 1/2, which motivates the following stronger conjecture.

\begin{conjecture}
There are functions $0 < g(n) < h(n)$ such that $\lim_{n \to \infty} h(n) =0$ with the following property.  If Chomp is to be randomly played with a starting configuration of $n$ cells and more than one row and more than one column, then the player going first will win with probability between $1/2 - h(n)$ and $1/2 - g(n)$.
\end{conjecture}

Proving this second conjecture (if in fact true) would presumably involve a judicious guess for the functions $g$ and $h$.  It should be noted that in light of Theorem \ref{chomp:corollary}, we have proven both of these conjectures for boards having at most 2 rows.  In that setting, for fixed $k$, we've shown $P(n,k) -1/2 \approx C_k n^{-2}$, but we weren't able to form any similar conjectures for the behavior of boards having 3 or more rows.

As an application of our results about random Chomp, we are also able to completely analyze random play in the game nim.  A description of this game can be found in Section \ref{section:notation}, but our result for this setting is as follows.

\begin{theorem}\label{theorem:nim}
Suppose two-player nim (either mis\`ere or normal play) is to be played starting with piles of sizes $s_1, s_2, \ldots, s_k$, such that each player's move is chosen independently and uniformly at random.
\begin{itemize}
\item[(a)] The expected number of moves played is given by $\displaystyle \sum_{i=1} ^{k} \sum_{r=1} ^{s_i} \dfrac{1}{r} = \sum_{i=1} ^{k} H_{s_{i}}$, where $\displaystyle H_t = \sum_{r=1} ^{t} \dfrac{1}{r}$ denotes the $t^{\text{th}}$ Harmonic number.
\item[(b)] If each pile has size $1$, then the game lasts exactly $k$ turns, and the winner depends only on whether $k$ is even or odd.
\item[(c)] Otherwise, if at least one pile has size more than 1, then each player wins with equal probability.
\end{itemize}
\end{theorem}

\subsection*{Outline of the paper}
We begin in Section \ref{section:notation} with brief discussion of general remarks and notation as well as a definition of the game nim.  In Section \ref{section:chomp expected value} we discuss expected game lengths in Chomp proving Theorems \ref{chomp:expected} and \ref{chomp:extremal shapes}.  Section \ref{section:chomp probability} then focuses on Chomp with at most two-rows proving Theorems \ref{chomp:probability} and \ref{chomp:corollary}.  We then analyze nim in Section \ref{section:nim} proving Theorem \ref{theorem:nim}.  Finally, we conclude in Section \ref{section:conclusion} with a brief discussion of possible directions for further research.

\section{General notation}\label{section:notation}
Throughout, all logs are understood to have base $e$.  We use $o(1)$ to denote a function tending to $0$ as the relevant parameter---typically $n$---tends to infinity.

\subsection*{Notation for Chomp}
Let $|B|$ denote the number of cells in the board $B$.  We write $D \lessdot B$ to mean that $D$ is a configuration obtainable from one move starting from $B$.  Since each cell of $B$ may be chosen for the next move, there are exactly $|B|$ boards $D$ for which $D \lessdot B$.  Say $\mathbb{E}[B]$ denotes the expected number of turns in a game starting with $B$, and $P(B)$ denotes the probability that the next player will win---or equivalently that the previous player will lose---in a game starting with $B$.

By conditioning on which cell of $B$ is selected, we obtain the recursive formula
\begin{equation}\label{eqn:chomp-expectation-recurrence}
\mathbb{E}[B] = 1 + \dfrac{1}{|B|} \sum_{D \lessdot B} \mathbb{E}[D],
\end{equation}
valid when $|B| > 0$.  As a base case, when $|B| = 0$, the empty board satisfies $\mathbb{E}[B] = 0$.

Similarly, if $P(B)$ denotes the probability that the player about to move will win---or equivalently that the previous player will lose---starting from $B$, then we have for all $|B| > 0$ that
\begin{equation}\label{eqn:chomp-probability-recurrence}
P(B) = \dfrac{1}{|B|} \sum_{D \lessdot B} \Big( 1 - P(D) \Big).
\end{equation}
As a base case, when $|B| = 0$, the empty board satisfies $P(B) = 1$.

\subsection*{Definition and notation for nim}
As introduced in 1901 by Bouton \cite{bouton1901nim}, nim is a two player game played with $k$ piles of tokens initially having $s_1, s_2, \ldots , s_k$ tokens in each pile.  The starting configuration is to be agreed upon before starting, and players alternate turns.  Each turn, a player selects a pile and then removes however many tokens they want from that pile (they must remove at least one token, but they may remove the entire pile if they so choose).  Play alternates, and the game ends when all the tokens have been removed.  In the rules of \textit{mis\`ere play} the player who removes the final token is the loser, but in \textit{normal play} that player is instead deemed the winner.  Our result in Theorem \ref{theorem:nim} is true in both \textit{mis\`ere} as well as \textit{normal} play, with the only distinction being that in this situations, the the parity condition in (b) is reversed.

To explicitly define the set of legal moves (which is necessary to define the randomly played version of the game), we will say that a player looking at a configuration of non-empty piles $(s_1, s_2, \ldots, s_k)$ must select one of the following moves:
\begin{itemize}
\item take 1 from pile 1, take 2 from pile 1, \ldots , take $s_1$ from pile 1;
\item take 1 from pile 2, take 2 from pile 2, \ldots , take $s_2$ from pile 2;
\item $\cdots$
\item take 1 from pile $k$, take 2 from pile $k$, \ldots , take $s_k$ from pile $k$.
\end{itemize}
Thus there are a total of $N = s_1 + s_2 + \cdots + s_k$ legal moves available.  The player selects pile $j$ with probability $s_j / N$, and having selected a pile, they remove a uniformly random amount from that pile.

That said, our analysis in Section \ref{section:nim} proves a somewhat more general version of Theorem \ref{theorem:nim} which remains valid merely under the assumption that after selecting a non-empty pile (via whatever probability distribution), the number of tokens removed from the selected pile is chosen uniformly at random.  Our statement of Theorem \ref{theorem:nim} and its subsequent proof are both carefully phrased so as to simultaneously address the cases of \textit{mis\`ere} and \textit{normal} play.

\section{Expected game length in Chomp}\label{section:chomp expected value}
As a helpful warm-up, we begin by proving Theorem \ref{chomp:expected} by induction.

\subsection*{Proof of Theorem \ref{chomp:expected}}
We prove this by induction on the number of cells in the Chomp board.  As a base case, if the board has only one cell, then the formula holds since any game with only one cell takes exactly one move, whereas the formula gives an expected value of $1/(1 \times 1) = 1$ (the case of an empty board requires $0$ moves, which also agrees with the formula under the convention that summing over the empty set yields 0).

Now suppose $B$ is an arbitrary Chomp board and that by induction, the formula holds for all configurations having fewer cells.  By appealing to \eqref{eqn:chomp-expectation-recurrence} and applying induction, we obtain
\[
\mathbb{E}[B] = 1 + \dfrac{1}{|B|} \sum_{D \lessdot B} \mathbb{E}[D] = 1 + \dfrac{1}{|B|} \sum_{D \lessdot B} \sum_{(x,y) \in D} \dfrac{1}{xy}.
\]
Finally, for any cell $(x,y) \in B$, we know there are exactly $|B| - xy$ boards $D \lessdot B$ for which $(x,y) \in D$.  This is because the cell $(x,y)$ would be \textit{missing} from $D$ iff $D$ is obtained from $B$ by chomping a cell $(p,q)$ where $p \leq x$ and $q \leq y$.  As there are exactly $x \cdot y$ such choices for $(p,q)$, there are precisely $|B| - xy$ boards $D \lessdot B$ for which $(x,y) \in D$.  Thus, by exchanging the order of summation we have

\begin{eqnarray*}
\mathbb{E}[B] &=& 1 + \dfrac{1}{|B|} \sum_{D \lessdot B} \mathbb{E}[D] = 1 + \dfrac{1}{|B|} \sum_{D \lessdot B} \sum_{(x,y) \in D} \dfrac{1}{xy}\\
&=& 1 + \dfrac{1}{|B|} \sum_{(x,y) \in B} \sum_{\substack{D\lessdot B, \\ D \ni (x,y)}}  \dfrac{1}{xy}\\
&=& 1 + \dfrac{1}{|B|} \sum_{(x,y) \in B} \dfrac{\# \{D \lessdot B \ : \ (x,y) \in D\} }{xy}\\
&=& 1 + \dfrac{1}{|B|} \sum_{(x,y) \in B} \dfrac{|B| - xy}{xy} = \sum_{(x,y) \in B} \dfrac{1}{xy}.
\end{eqnarray*}
Thus, the formula holds for $B$, completing the induction proof. $\qed$

\subsection*{Proof of Theorem \ref{chomp:extremal shapes}}
\paragraph*{Lower bound:} Consider the \textit{lexicographic} ordering on cells $(x,y)$---namely, $(a,b) \preceq_{lex} (c,d)$ iff either $a < c$ or $a=c$ and $b \leq c$.  Then for each cell $(x,y) \in B$, define its \textit{rank} to be the number of cells $(p,q) \in B$ such that $(p,q) \preceq_{lex} (x,y)$.

We claim that for each $(x,y) \in B$, we must have $rank(x,y) \geq xy$.  To see this, note that if $(x,y) \in B$, then for each $(p,q)$ where $p \leq x$ and $q \leq y$, it must be the case that $(p,q) \in B$ as well.  Moreover, there are exactly $xy$ such values of $(p,q)$ in $B$, and each one lexicographically precedes [or is equal to] $(x,y)$.  Thus, $rank(x,y) \geq xy$.

From this, we have
\[
\mathbb{E}[B] = \sum_{(x,y) \in B} \dfrac{1}{xy} \geq \sum_{(x,y) \in B} \dfrac{1}{rank(x,y)} = \sum_{r=1}^{|B|} \dfrac{1}{r} = H_{|B|},
\]
which is the expected number of turns required when starting with the board consisting of a single row having $|B|$ cells.

\paragraph*{Upper bound:} For a quick upper bound, note that the board $B$ has at most $|B|$ rows, and each of row has at most $|B|$ columns.  Thus
\[
\mathbb{E}[B] = \sum_{(x,y) \in B} \dfrac{1}{xy} \leq \sum_{x=1} ^{|B|} \sum_{y=1} ^{|B|} \dfrac{1}{xy} = \left ( \sum_{r=1} ^{|B|} \dfrac{1}{r} \right)^2.
\]

To obtain a tight upper bound, we need to define the family of boards for which $\mathbb{E}[B]$ is maximized (given the number of cells).  To this end, for each $t$, define $\mathcal{P}_t = \{ (x,y) \in \mathbb{Z}^2 \ : xy \leq t, \ \ \ 1 \leq x, \ \ 1 \leq y \}$.  Let $\tau$ be the integer for which $|\mathcal{P}_\tau| \leq |B| < |\mathcal{P}_{\tau+1}|$, and let $B'$ denote any board having $|B|$ cells for which $\mathcal{P}_\tau \subseteq B' \subset \mathcal{P}_{\tau+1}.$  Clearly, by considering our formula from Theorem \ref{chomp:expected}, each such board has the same expected game length since every cell $(x,y) \in \mathcal{P}_{\tau+1} \setminus \mathcal{P}_{\tau}$ satisfies $xy = \tau+1$.  It is also clear that the board $B'$ has the largest expected game length among all boards with $|B|$ cells (since $B'$ is constructed greedily so that every cell of $B'$ contributes as much as possible to the expected value by maximizing $1/xy$).  So to finish the proof of the upper bound, we need only estimate $\mathbb{E}[B']$.

For this, note
\[
\mathbb{E}[B] \leq \mathbb{E}[B'] = \dfrac{|B' \cap \mathcal{P}_{\tau+1}|}{\tau +1} + \sum_{(x,y) \in \mathcal{P}_{\tau}} \dfrac{1}{xy} = \dfrac{|B' \cap \mathcal{P}_{\tau+1}|}{\tau +1} + \sum_{xy \leq \tau} \dfrac{1}{xy}.
\]

These quantities are well-known in analytic number theory, and  we direct the curious reader to Apostol's text \cite{apostol} for a discussion of this---including the notation to follow.  For example, exercise 2 of chapter 3 of this text proves
\[
\sum_{xy \leq t} \dfrac{1}{xy} = \sum_{m=1} ^{t} \dfrac{d(m)}{m} = \log(t)^2/2 + 2\gamma \log(t) + \mathcal{O}(1),
\]
where $\gamma = 0.577\ldots$ denotes the Euler-Mascheroni constant (and $d(m)$ denotes the divisor function).  Since $|B' \cap P_{\tau +1}| \leq |P_{\tau+1} \setminus P_{\tau}| = d(\tau +1)$, we have $0 \leq |B' \cap P_{\tau+1}| / (\tau +1) \leq \dfrac{d(\tau+1)}{\tau +1} \leq 1$, and thus we obtain
\[
\mathbb{E}[B'] =  \dfrac{|B' \cap \mathcal{P}_{\tau+1}|}{\tau +1} + \sum_{xy \leq \tau} \dfrac{1}{xy} = \log(\tau)^2/2 + 2\gamma \log(\tau) + \mathcal{O}(1).
\]

Since $\mathcal{P}_{\tau} \subseteq B' \subset \mathcal{P}_{\tau+1}$, it follows that $|\mathcal{P}_{\tau}| \leq |B| < |\mathcal{P}_{\tau+1}|$.  And by appealing to Theorem 3.3 of \cite{apostol}, we have
\[
|\mathcal{P}_t| = \sum_{(x,y), xy \leq t} 1 = \sum_{m=1} ^{t} d(m) = t \log(t) + (2 \gamma -1) t + \mathcal{O}(\sqrt{t}).
\]
Thus, $|B| = \tau \log(\tau) + (2\gamma - 1) \tau + \mathcal{O}(\sqrt{\tau})$, meaning $\tau \sim |B| / \log(|B|)$, which proves
\[
\log(\tau) = \log(|B|) - \log \log(|B|) + o(1),
\]
establishing the estimate $\mathbb{E}[B'] = (1/2 + o(1)) \log(|B|)$.  And using the well-known fact that the Harmonic numbers satisfy $H_N = \sum_{r=1} ^{N} \frac{1}{r} = (1+o(1))\log(N)$ completes the proof (see, e.g., Theorem 3.2 of \cite{apostol}).  $\qed$

\section{Two-rowed Chomp: win probability}\label{section:chomp probability}
We now turn our attention to proving Theorems \ref{chomp:probability} and \ref{chomp:corollary}.  Throughout, let $P(n,k)$ denote the probability that the next player will win---equivalently that the previous player will lose---starting from a board having $n$ cells in its first row and $k$ in its second.

\subsection*{Proof of Theorem \ref{chomp:probability}}
We essentially prove this by induction, with $k=0$ as a special case.

\paragraph*{Formula for $k=0$:} First note that if $B$ is the empty board, then the game has ended and the previous player has lost.  Thus $P(0,0) = 1$.  Moreover, if $n \geq 1$, then the recursive formula \eqref{eqn:chomp-probability-recurrence} yields
\[
P(n,0) = 1 - \dfrac{1}{n} \Big(\mathbb{P}(P(0,0) + \cdots + P(n-1,0) \Big).
\]
Therefore, we obtain $P(1,0) = 1 - (1/1) (P(0,0)) = 1 - 1 = 0$.  And for all $n \geq 2$ we therefore have
\begin{eqnarray*}
P(n,0) &=& 1 - \dfrac{1}{n} \Big(\mathbb{P}(P(0,0) + \cdots + P(n-1,0) \Big)\\
&=& \dfrac{1}{n} \left(0 + 1 + \sum_{j=2} ^{n-1} \Big( 1-P(j,0) \Big) \right).
\end{eqnarray*}
Now by induction on $n$ (with $k=0$ fixed), we can assert that for each $j \in \{2, 3, \ldots, n-1\}$ we know $P(j,0) = 1/2$, which would in turn yield
\begin{eqnarray*}
P(n,0) &=& \dfrac{1}{n} \left(0 + 1 + \sum_{j=2} ^{n-1} \Big( 1-\mathbb{P}(W(B_{j,0})) \right)\\
&=& \dfrac{1}{n} \left(1 + \sum_{j=2} ^{n-1} 1/2 \right) = 1/2,
\end{eqnarray*}
thus proving that for all $n \geq 2$ we have $P(n,0) = 1/2$, as desired.

\paragraph*{Formula for $k \geq 1$:}
For potential ease of presentation, we'll prove that our formula works by contradiction.  In fact, suppose $(n,k)$ is a point where the formula fails, and suppose $(n,k)$ is chosen (subject to the constraint that $n \geq k \geq 1$) to be a minimal counterexample in the sense that if we have any other point $(x,y)$ where $x \geq y \geq 1$ and either $x < n$ or $y < k$, then the stated formula holds for $(x,y)$.

Again, by \eqref{eqn:chomp-probability-recurrence} we have
\begin{eqnarray*}
P(n, k) &=& 1 - \dfrac{1}{n+k} \left( \sum_{j=0} ^{k-1} P(j,j) + \sum_{j=k} ^{n-1} P(j,k) + \sum_{j=0} ^{k-1} P(n,j) \right)\\
&=& \dfrac{1}{2} - \dfrac{1}{n+k} \Bigg([P(0,0) + P(n,0)] - (1) \\
& & \qquad \qquad \qquad + \sum_{j=1} ^{k-1} [P(j,j)-1/2] + \sum_{j=1} ^{k-1} [P(n,j) -1/2] \\
& & \qquad \qquad \qquad \qquad + \sum_{j=k} ^{n-1} [P(j,k)-1/2] \Bigg).
\end{eqnarray*}
Using the above (together with the proven facts that $P(0,0) = 1$ and $P(1,0) = 0$) shows that $P(1,1) = 1/2$ as required.  Thus, we may assume that $(n,k) \neq (1,1)$, meaning $n \geq 2$.  Furthermore, by the minimality of $(n,k)$---equivalently, by an appropriately stated induction---and the fact that $P(n,0) = 1/2$ (since $n \geq 2$) we have
\[
P(n, k) = \dfrac{1}{2} - \dfrac{1}{n+k} \left( \dfrac{1}{2} - \sum_{j=1} ^{k-1} C(j,j) - \sum_{j=1} ^{k-1} C(n,j) - \sum_{j=k} ^{n-1} C(j,k) \right),
\]
where $C(x,y) = \dfrac{x \alpha_y + \beta_y}{(2(y-1))! \cdot (x+y)(x+y-1)(x+y-2)}$ whenever $(x,y) \neq (1,1)$ and $C(1,1) = 0$.
But we may then apply Lemma \ref{lemma:induction}---stated and proven below---to conclude
\[
P(n, k) = \dfrac{1}{2} - \dfrac{1}{n+k} \left( \dfrac{n \alpha_k + \beta_k}{(2(k-1))! \cdot (n+k-1)(n+k-2)} \right),
\]
and so $P(n,k)$ also satisfies the required formula, proving Theorem \ref{chomp:probability}.

\begin{lemma}\label{lemma:induction}
For any real $x > 1$ and any integer $y \geq 1$ define
\[
C(x,y) = \dfrac{x \alpha_y + \beta_y}{(2(y-1))! \cdot (x+y)(x+y-1)(x+y-2)}
\]
and $C(1,1) = 0$.  For all integers $n \geq k \geq 1$, if $(n,k) \neq (1,1)$ then
\[
(n+k) C(n,k) = \dfrac{1}{2} - \sum_{j=1} ^{k-1} C(j,j) - \sum_{j=1} ^{k-1} C(n,j) - \sum_{j=k} ^{n-1} C(j,k).
\]
\end{lemma}
\begin{proof}
We first prove the lemma when $k=1$, which is easily done by induction.  As a base case, when $(n,k) = (2,1)$ we have---using the fact that $\alpha_1 = 1$ and $\beta_1 = -1$---that $(2+1) C(2,1) = 3/6 = 1/2$, whereas the expression on the left-hand-side is simply $1/2$.  Furthermore, if $n > 2$, the left-hand-side of the desired formula is
\[
\dfrac{1}{2} - \sum_{j=1} ^{n-1} C(j,1) = -C(n-1, 1) + \dfrac{1}{2} - \sum_{j=1} ^{n-2} C(j,1).
\]
By induction, this is equal to $-C(n-1,1) + (n-1 +1) C(n-1,1)$, which is $(n-1) C(n-1, 1)$.  And since $(n-1) C(n-1,1) = 1/n = (n+1)C(n,1)$, this proves the lemma holds when $k=1$.

Now let $k \geq 2$, and for all $n > 0$ define
\[
g(n,k) = \dfrac{(n-k) \Big[\alpha_k (k^2 + 3k (n-1) -2n+2)  + \beta_k (3k+n-3) \Big]}{4(k-1) (2k-1)(2 (k-1))! (n+k-1)(n+k-2)}.
\]
We claim that for all integers $k \geq 2$ and every integer $n \geq k$ we have

\[
g(n,k) = \sum_{j=k} ^{n-1} C(j,k).
\]
In fact, this too is readily proven by induction on $n$ (letting $k \geq 2$ be fixed but arbitrary).  As a base case, we see from its definition that $g(k,k) = 0$, which agrees with the empty summation.  And the induction step follows immediately after using the definition of $g$ to symbolically verify the fact that $g(n+1,k) - g(n,k) = C(n+1, k)$.

Finally, we will prove that for all integers $k \geq 2$ and for all real numbers $n > 1$ we have
\[
(n+k) C(n,k) = \dfrac{1}{2} - \sum_{j=1} ^{k-1} C(j,j) - \sum_{j=1} ^{k-1} C(n,j) - g(n,k).
\]
For this, let $n > 1$ be fixed but arbitrary, and we proceed by induction on $k$.  As a base case, when $k=2$, the claim holds by directly verifying $(n+2)C(n,2) = 1/2 - g(n,2)$ from the definitions.  As for the induction step, we proceed by symbolically verifying $(n+k+1)C(n,k+1) - (n+k)C(n,k)$ equals $-C(k,k) - C(n,k) - g(n,k+1) + g(n,k)$.  Doing so is in fact immediate after rewriting each of the terms $\alpha_{k+1}$ and $\beta_{k+1}$ appearing in $g(n,k+1)$ as an appropriate linear combination of $\alpha_k$ and $\beta_k$ by using the defining recurrence relation for $\alpha$ and $\beta$.
\end{proof}

\subsection*{Proof of Theorem \ref{chomp:corollary}}
We first prove for all $k \geq 2$ that $0 < \alpha_k \leq -\beta_k \leq (k-1) \alpha_k$, done by induction on $k$.  The base case follows since $\alpha_2 = 4$ and $\beta_2 = -4$.

Now for the induction step, we have
\[
\alpha_{k+1} = (4k^2 + k) \alpha_k + \beta_k \geq (4k^2 + k) \alpha_k - (k-1) \alpha_k = (4k^2+1) \alpha_k > 0.
\]
Moreover, 
\begin{eqnarray*}
\alpha_{k+1} &=& (4k^2 + k) \alpha_k + \beta_k \leq (k^2 + k) \alpha_k +3k^2 \alpha_k + \beta_k\\
&\leq& (k^2 + k) \alpha_k -3k^2 \beta_k + \beta_k = -\beta_{k+1} + (k^2 + k)\beta_{k} \leq -\beta_{k+1}.
\end{eqnarray*}
And finally
\begin{eqnarray*}
-\beta_{k+1} &=& k(k+1) \alpha_k - (4k^2)\beta_k +  (k+1) \beta_k\\
&\leq& k(k+1) \alpha_k + (4k^2)(k-1) \alpha_k +  (k+1) \beta_k\\
&=& k \alpha_{k+1} - k(4k-1) \alpha_k + \beta_k \leq k \alpha_{k+1}.
\end{eqnarray*}
Thus we have $0 < \alpha_{k+1} \leq -\beta_{k+1} \leq k \alpha_{k+1}$, which proves by induction that these inequalities hold for all $k$.

Now for each $k \geq 1$, we know $\alpha_{k+1} = (4k^2+k) \alpha_k + \beta_k$.  Thus
\[
4 k^2 \alpha_k \leq 4(k+1)k \alpha_{k} - (k-1) \alpha_k \leq \alpha_{k+1} \leq 4(k+1)k\alpha_k,
\]
which means for all $k \geq 1$ that $4^{k-1} [(k-1)!]^2 \leq \alpha_k \leq 4^{k-1} [(k-1)!]^2 \cdot k$.

Finally, we need only bound $P(n,k)$.  For this, it is easy to see the desired upper bounds on $P(n,0)$ and $P(n,1)$ simply by appealing to Theorem \ref{chomp:probability}.  We first note that for all $k \geq 1$ we have
\[
\dfrac{4^{k-1} [(k-1)!]^2}{(2(k-1))!} \leq \dfrac{\alpha_k}{(2(k-1))!} \leq k \dfrac{4^{k-1} [(k-1)!]^2}{(2(k-1))!}.
\]
And by using a standard bound on the central binomial coefficients (e.g., Proposition 3.6.2 of \cite{matouvsek2009invitation}), we further see for all $k \geq 2$ that
\[
\sqrt{k} \leq \sqrt{2 (k-1)} \leq \dfrac{4^{k-1} \cdot [(k-1)!]^2}{(2(k-1))!} \leq 2 \sqrt{k-1} \leq 2 \sqrt{k}.
\]
Combining this $\beta_k < 0$ yields by Theorem \ref{chomp:probability} that for all $k \geq 2$,
\begin{eqnarray*}
\dfrac{1}{2} - P(n,k) &=& \dfrac{n \alpha_k + \beta_k}{(2(k-1))! \cdot (n+k)(n+k-1)(n+k-2)}\\
&\leq& \dfrac{n \alpha_k}{(2(k-1))! \cdot (n+k)(n+k-1)(n+k-2)}\\
&\leq& \dfrac{2n k^{1.5}}{(n+k)(n+k-1)(n+k-2)}.
\end{eqnarray*}
Similarly, using that $-\beta_k \leq (k-1) \alpha_k$, we obtain
\begin{eqnarray*}
\dfrac{1}{2} - P(n,k) &=& \dfrac{(n-k+1) \alpha_k + (k-1)\alpha_k + \beta_k}{(2(k-1))! \cdot (n+k)(n+k-1)(n+k-2)}\\
&\geq& \dfrac{(n-k+1) \alpha_k}{(2(k-1))! \cdot (n+k)(n+k-1)(n+k-2)}\\
&\geq& \dfrac{(n-k+1) \sqrt{k}}{(n+k)(n+k-1)(n+k-2)},
\end{eqnarray*}
which completes the proof. $\qed$

\section{Analyzing nim: proof of Theorem \ref{theorem:nim}}\label{section:nim}
To prove Theorem \ref{theorem:nim}, we could simply appeal to an induction argument as we did in the case of Chomp, but it would be faster (and likely clearer) to instead derive Theorem \ref{theorem:nim} from the results we've already proven about Chomp.

Let us consider an arbitrary nim configuration with piles of size $s_1, s_2, \ldots, s_k$.  For each pile $i$, let $T_i$ denote the total number of times that pile $i$ was selected before it was depleted.  Then clearly the total number of turns required is $T = T_1 + T_2 + \cdots + T_k$.  Moreover, the variables $\{T_1, T_2, \ldots, T_k\}$ are independent.

For each $i$, consider a random game of Chomp starting with a board having $s_i$ cells all in the same row.  Let $C_i$ denote the random variable for how many turns would be required in that game.  Then each $T_i$ has the same distribution as $C_i$ because one-pile-nim and one-row-Chomp are the same game.\footnote{The only distinction is that the ``winners" in these games might be different depending on whether nim is being played under the mis\`ere or normal rule.}

Thus the number of turns in this nim game has the same distribution as the sum of $k$ independent one-row-Chomp game lengths.

\paragraph*{Expected length:} From this (and using $\mathbb{E}[T_i] = \mathbb{E}[C_i] = H_{s_i}$), we immediately see that the expected number of turns in nim is equal to
\[
\mathbb{E}[T] = \mathbb{E}\left[ \sum_{i=1} ^{k} T_i \right] = \sum_{i=1} ^{k} \mathbb{E}\left[ T_i \right] = \sum_{i=1} ^{k} \left[  \sum_{n=1} ^{s_i} \dfrac{1}{n}\right] = \sum_{i=1} ^{k} H_{s_i}.
\]

\paragraph*{Win probabilities:} The win probabilities are entirely determined by the parity of $T$ (with ``even" and ``odd" corresponding to who wins according to whether the mis\`ere or normal rule is in effect).  For that, note that if each pile has size $1$, then each $T_i$ is deterministically equal to $1$, so $T = k$, proving statement (b) of Theorem \ref{theorem:nim}.

On the other hand, if some pile $T_{p}$ has size more than 1, then by Theorem \ref{chomp:probability} (for the case of one-row-Chomp), the variable $T_p$ is equally likely to be even as to be odd.  Thus, since the variables $\{T_1, T_2, \ldots, T_{k}\}$ are independent, $T = T_1 + T_2 + \cdots + T_k$ is \textit{also} uniform modulo 2.  Therefore, such a game of nim would have both players equally likely to win, proving statement (c) of the theorem. $\qed$

\section{Concluding remarks}\label{section:conclusion}
Perhaps the most obvious direction for future work would be to attempt generalizing Theorems \ref{chomp:probability} and \ref{chomp:corollary} to Chomp boards with more than two rows, with Conjectures 1 and 2 being particularly interesting.  Another natural line of questioning would be to extend this work to the study of other combinatorial games played randomly (e.g., hackenbush or various subtraction games).  In those settings, we would immediately have recursive formulas such as \eqref{eqn:chomp-expectation-recurrence} and \eqref{eqn:chomp-probability-recurrence} for expected values and win probabilities, and in fact the question amounts to studying the process of a simple random walk done on the directed graph of game states.

Another possible avenue of study could be to ask other questions about these random processes.  For instance, if $T$ denotes the random variable for how many turns the random game will last, then we studied $\mathbb{E}[T]$ as well as $\mathbb{P}(\text{$T$ is even})$, which is exactly the probability of a given player winning.  But other questions about the distribution of $T$ would be interesting as well.  For Chomp---even with only two rows---the limiting distribution of $T$ would be quite interesting to determine.

\paragraph*{Acknowledgments.}
This work was done as a part of the summer undergraduate research program at Swarthmore College with additional funding for the second author provided through the Frances Velay Women’s Science Research Fellowship Program.

\bibliographystyle{plain}
\bibliography{references.bib}
\end{document}